\documentclass{amsart}

\usepackage{geometry}                
\usepackage[applemac]{inputenc}
\usepackage[T1]{fontenc}
\usepackage{bbm}
\usepackage{comment}
\usepackage{xcolor}
\usepackage{tikz-cd}
\usepackage{amsmath,amsfonts,amssymb,amsthm}
\usepackage{enumitem}
\usepackage{caption}
\setlist[enumerate]{wide=0pt, widest=99,leftmargin=\parindent, labelsep=*}
\setlist[enumerate]{align=left}
\usepackage{mathabx}

\theoremstyle{plain}
\newtheorem{Theorem}{Theorem}[section]
\newtheorem{Proposition}[Theorem]{Proposition}
\newtheorem{Lemma}[Theorem]{Lemma}
\newtheorem{Corollary}[Theorem]{Corollary}
\newtheorem{Scholium}[Theorem]{Scholium}
\newtheorem{Fact}[Theorem]{Fact}
\theoremstyle{definition}
\newtheorem{Definition}[Theorem]{Definition}
\newtheorem{Example}[Theorem]{Example}

\theoremstyle{remark}
\newtheorem{Remark}[Theorem]{Remark}

\DeclareMathOperator{\Ob}{Ob}

\newcommand{\tto}{\longrightarrow}

\newcommand{\C}{\mathcal{C}}

\newcommand{\Set}{\mathbf{Set}}

\newcommand{\vFld}{\mathbf{vFld}}
\newcommand{\vHyp}{\mathbf{vHyp}}

\title{A result of Krasner in categorial form}

\author{Alessandro Linzi}

\address{Centre for Information Technologies and Applied Mathematics, University of Nova Gorica \hbox{(Slovenia).}}

\email{alessandro.linzi@ung.si}

\keywords{Valuation; completion; field; hyperfield; category-theoretic algebra.} 

\thanks{The author would like to express his gratitude to Franz-Viktor  Kuhlmann and Pierre Touchard for the numerous discussions had with the author on the topic of the paper. Many thanks also to Irina Cristea for her invaluable support.}

\begin{document}

\maketitle

\begin{abstract}
In 1957 M.\ Krasner described a complete valued field $(K,v)$ via the projective limit of a system of certain structures, called hyperfields, associated to $(K,v)$. We put this result in purely category-theoretic terms by translating into a limit construction in certain slice categories of the category of valued hyperfields and their homomorphisms.
\end{abstract}

\section{Introduction}

If one considers the operations of classical algebraic structures (such as groups, rings, fields,\ldots) by looking at their graphs, then one sees that they satisfy two fundamental assumptions: they are left total and functional. In other words, these properties can be spelled out as operations being \emph{everywhere defined} (i.e., the operation can be applied to any two elements to obtain at least one result) and \emph{single-valued} (i.e., an application of the operation to any two elements yields at most one result). A hyperfield is a field-like structure where the latter property is relaxed for the additive operation. In the literature, such structures appear perhaps more than one would expect: hyperfields are of interest e.g.\ in tropical geometry \cite{Vir10,Jun21,JSY22}, symmetrization \cite{Hen13,Jun18,Row22}, projective geometry \cite{CC11}, valuation theory \cite{KLS22,Lee20,LT22,Gla17} and ordered algebra \cite{Gla10,GM12,LKS23}. There are even reasons to believe that their theory generalises field theory in ways that can be used to tackle deep problems such as the description of $\mathbb{F}_1$, the "field of characteristic one" (cf.\ \cite{Tit57,CC11}). 

More generally, since the pioneer papers \cite{Mar34,Mar35,Mar36} of F.\ Marty, structures with multivalued operations (also called hyperstructures) generalising classical singlevalued structures have been the object of several research projects. For example, modules with a multivalued operation and a scalar multiplication over Krasner hyperrings (the old brothers of hyperfields) have been studied e.g.\ in \cite{NC17,BCK20,BC22}. In addition, hypergroups are mentioned in a journal of theoretic physics, in \cite{JJ17}.

In this paper, we focus on valued hyperfields, generalising valued fields: the object of study of classical valuation theory. To any valued field $(K,v)$, Krasner's quotient construction associates a projective system $\mathcal{H}$ of valued hyperfields, indexed by the non-negative elements of the value group. The projective limit of the system $\mathcal{H}$ is a valued field isomorphic (as a valued field) to the valuation-theoretic completion of $(K,v)$. This result was proved by Krasner in \cite{Kra57} to achieve an approximation result for complete valued fields of positive characteristic by means of other fields of characteristic $0$. His proof, makes heavy use of the ultrametric structure of the hyperfields in $\mathcal{H}$ induced by the canonical one of $(K,v)$. On the other hand, it can be noted that, while valuation maps generalise readily to the setting of hyperfields, a valuation on a hyperfield needs not to induce an ultrametric distance as in the singlevalued case. Krasner's broader aim led him to prefer the approach of declaring \lq\lq valued\rq\rq\ only those hyperfields which admit a valuation map inducing (in a prescribed way) an ultrametric on the underlying set. However, examples of hyperfields with interesting valuation maps which do not induce an ultrametric as required by Krasner are now known (see e.g.\ \cite[Example 4.3]{KLS22}). Among these, \emph{generalised tropical hyperfields} \cite[Example 2.14]{Lin23} happen to encode ordered abelian groups as valued hyperfields. 

In addition, without the ultrametric condition, valuations on (hyper)fields become nothing but homomorphisms with a specific target, making the category $\vHyp$ of valued hyperfields and their homomorphisms a natural framework (also) for classical valuation theory.

In this article, we propose the question whether the metric properties of the members and of the projections of the system associated to valued fields (such as $\mathcal{H}$), can somehow be captured in the just mentioned category $\vHyp$, while respecting the principle of equivalence.

We show in fact that the metric properties of the systems $\mathcal{H}$ are \emph{local} in $\vHyp$, in the sense that they can be deduced by seeing $\mathcal{H}$ as a diagram in the slice categories of $\vHyp$ by (suitable) generalised tropical hyperfields. There, Krasner's result goes through as a clean limit construction and valuation-theoretic completions of valued fields are hence described as vertices of such limit cones.

We organised the paper by first explaining briefly the necessary category-theoretic backgroung and the notation that we adopt to treat it. Then we survey the algebraic theory of hyperfields as well as Krasner's quotient construction associating to a valued field the projective system $\mathcal{H}$ of valued hyperfields mentioned above. We define the category $\vHyp$ of valued hyperfields and explain how ordered abelian groups (OAGs) can be interpreted as objects in it. Then, by passing to slice categories over OAG-extensions of the value group of a valued field $(K,v)$, we show how its associated system forms a diagram in those slice categories and that the limit of this diagram is isomorphic as a valued field to the completion of $(K,v)$. 

\section{Category Theoretic Preliminaries and Terminology}

Many references for category theory may be cited, which cover the necessary background for the scope of this paper. The classic \cite{ML71} is certainly one of them and we found particularly useful the following books as well \cite{Awo10, Sim11, Lei14}. 

Since we believe that one way to better appreciate our contribution is to be as precise as possible with terminology and notation, we shall assume only some familiarity with the concepts of category, small category, and functor during this preliminary section. 

As for basic notations, for a category $\C$, we write $A\in\Ob(\C)$ to mean that $A$ is an object in $\C$ and, for any ordered pair $(A,B)$ of objects in $\C$, we denote the set of arrows $f:A\tto B$ in $\C$ by $\C(A,B)$. For the composition of arrows the symbol $\circ$ will be employed and the identity $\C$-arrow of $A\in\Ob(\C)$ will be written as $1_A:A\tto A$.

If $C$ is an object in $\C$, then the \emph{slice} category $\C/C$ has $\C$-arrows $f:A\tto C$, with $A\in\Ob(\C)$, as objects, while a $\C$-arrow $a:A\tto B$ is an arrow 
\[
\begin{tikzcd}[column sep=large,every label/.append style = {font = \small}]
A\arrow[d,"~f"{name=F}, near start]&&B\arrow[d,"g~"' {name=G}, near start]&\\
C&&C
\ar[from=F,to=G,shorten=2mm]
\end{tikzcd}
\] 
in $\C/C$ if and only if the following triangular diagram:
\[
\begin{tikzcd}[column sep=huge, row sep=small]
A\arrow[dd,"a"']\arrow[dr,"f"]&\\
&C\\
B\arrow[ur,"g"']&
\end{tikzcd}
\]
is commutative in $\C$, i.e., $g\circ a=f$.

The concept of limit in a category is fundamental and, partly due to its generality, admits many equivalent definitions. The terminology we adopt for limits is similar to that of \cite{Lei14}, which seemed to be the most appropriate in this case. Let us go through a brief recap.

\subsection{Limits}

Fix a category $\C$. If $\mathcal{S}$ is a small category, then a functor $D:\mathcal{S}\tto\C$ is called a \emph{diagram in $\C$ of shape $\mathcal{S}$}. A \emph{cone} on a diagram $D:\mathcal{S}\tto\C$ consists of an object $V$ in $\C$, called the \emph{vertex} of the cone, together with a family, indexed by the collection of objects in $\mathcal{S}$,
\begin{equation}\label{cone}
\begin{tikzcd}[sep=huge]
\bigl(~V\arrow[r,"s_I"]& D(I)~\bigr)_{I\in\Ob(\mathcal{S})}
\end{tikzcd}
\end{equation}
of $\C$-arrows, called the \emph{sides} of the cone such that the following triangular diagram:
\[
\begin{tikzcd}[column sep=huge, row sep=small]
&D(I)\arrow[dd,"D(f)"]\\
V\arrow[dr,"s_J"']\arrow[ur,"s_I"]&\\
&D(J)
\end{tikzcd}
\]
is commutative, for all arrows $f\in \mathcal{S}(I, J)$.

A cone over a diagram $D:\mathcal{S}\tto\C$,
\[
\begin{tikzcd}[sep=huge]
\bigl(~L\arrow[r,"p_I"]& D(I)~\bigr)_{I\in\Ob(\mathcal{S})}
\end{tikzcd}
\]
is called a \emph{limit cone} if it satisfies the following \emph{universal property}: 
\begin{quote}
for any cone on $D$ as in \eqref{cone}, there exists a unique arrow $h:V\longrightarrow L$ such that $p_I\circ h=s_I$ holds, for all $I\in\Ob(\mathcal{S})$.
\end{quote}
By a \emph{limit} of a diagram $D:\mathcal{S}\tto\C$ we mean the vertex of a limit cone over $D$.

An \emph{isomorphism} between objects $A,B$ in a category $\C$ is a $\C$-arrow $f:A\tto B$ with the property that a $\C$-arrow $f^{-1}:B\tto A$ exists such that $f^{-1}\circ f=1_A$ and $f\circ f^{-1}=1_B$. The notation $f^{-1}$ used for this arrow suggests that from the mere existence, uniqueness follows too, which is in fact well-known to be the case.

The uniqueness of the arrows whose existence is guaranteed by the universal property of limit cones, implies that, when they exist, limit cones are \emph{unique up to }(a unique) \emph{isomorphism} (of cones). In particular, limits in a category $\C$ are unique up to (a unique) $\C$-isomorphism. It is up to this isomorphism that we speak of \emph{the} limit of a diagram.

The sides of limit cones are often called \emph{projections}. This name comes from the analogy with the limit of diagrams of shape $\mathbf{2}$, that is, the category consisting of $2$ objects with their identity arrows solely. The latter specially shaped limits are called \emph{(binary) products}. In fact, in the category $\Set$ of sets and functions, their vertex is the familiar cartesian product of sets, while their sides are nothing but the projections onto its components. For the product $(A\times B,p_1,p_2)$ of two objects $A_1,A_2\in\Ob(\C)$, where $p_i:A_1\times A_2\tto A_i$ denote the projections ($i=1,2$), the universal property of limits has the following form:
\begin{quote}
for any object $B$ in $\C$ admitting two arrows $f_i:B\tto A_i$ ($i=1,2$) in $\C$, there exists a unique arrow $f_1\times f_2:B\tto A_1\times A_2$ such that the following diagram:
\[
\begin{tikzcd}[column sep=huge, row sep=small]
&&A_1\\
B\arrow[r,"f_1\times f_2"]\arrow[urr,bend left,"f_1"]\arrow[drr,bend right,"f_2"']&A_1\times A_2\arrow[dr,"p_2"']\arrow[ur,"p_1"]&\\
&&A_2
\end{tikzcd}
\]
is commutative.
\end{quote}

Another specially shaped limit, which is named \emph{terminal object}, is defined in a category $\C$ as the limit cone of the unique diagram $\emptyset\tto\C$, where $\emptyset$ denotes the category with no objects and, consequently, no arrows (the \emph{empty category}). If $T$ is a terminal object in $\C$, then the universal property of limits has the following form:
\begin{quote}
for any object $C$ in $\C$, there exists a unique arrow $!_C:C\longrightarrow T$.
\end{quote}

When limit cones on diagrams of a certain shape $\mathcal{S}$ exist in a category $\C$, then one says that $\C$ \emph{has limits of shape} $\mathcal{S}$. One then usually simplifies the terminology further in case the particular shape has been given a name. For instance, phrases like \lq\lq$\C$ has products\rq\rq\ or \lq\lq$\C$ has a terminal object\rq\rq\ mean that $\C$ has limits of shape $\mathbf{2}$ and $\emptyset$, respectively.

\begin{Remark}
It is important to keep in mind that uniqueness up to isomorphism does not necessarily mean absolute uniqueness\footnote{Following the remark on terminology in \cite[Page x]{Gol84}, one may phrase this as \lq\lq categorial uniqueness is not categorical\rq\rq.}. For example, in the category $\Set$ of sets and functions, where isomorphisms are bijections, all singleton sets are terminal objects.
\end{Remark}

\section{Valued Fields and Hyperfields}

Let $(K,+,\cdot,0,1)$ be a field and $(\Gamma,\leq,+,0)$ a linearly ordered abelian group (always denoted additively)\footnote{Note that we use the same symbols to denote the additive structure of $K$ and the abelian group structure of $\Gamma$. This is standard practice and will cause no confusion.}. That is, $\Gamma$ is an abelian group equipped with a linear order relation $\leq$ and an abelian group structure whose operation $+$ is compatible with $\leq$, i.e., the following implication:
\[
\gamma\leq\delta\quad\implies\quad \gamma+\varepsilon\leq\delta+\varepsilon
\] 
holds, for all $\gamma,\delta,\varepsilon\in\Gamma$. A map $v:K\tto\Gamma\cup\{\infty\}$, where $\infty$ is a symbol such that $\gamma+\infty=\infty+\gamma=\infty>\gamma$ for all $\gamma\in\Gamma$, is called a \emph{(Krull) valuation on }$K$ if and only if it satisfies all of the following three properties:
\begin{enumerate}
\item[(VAL1)] $v(x)=\infty$ if and only if $x=0$, for all $x\in K$.\label{first-item}
\item[(VAL2)] $v(xy)=v(x)+v(y)$, for all $x,y\in K$.
\item[(VAL3)] $v(x+y)\geq\min\{v(x),v(y)\}$, for all $x,y\in K$.\label{last-item}
\end{enumerate}

If a valuation $v$ on a field $K$ is given, then $(K,v)$ is called a \emph{valued field}, while the image of $v$ in $\Gamma$, denoted by $vK$, is called the \emph{value group of }$(K,v)$. The \emph{value} $v(x)$ of $x\in K$ will be written as $vx$ whenever no risk of confusion arises. If $(K,v)$ is a valued field, then
\[
\mathcal{O}_v:=\{x\in K\mid vx\geq 0\}
\]
is a subring of $K$, called the \emph{valuation ring} of $(K,v)$. It determines the valuation map $v$ up to valuation-equivalence, i.e., up to composition with an order preserving isomorphism of the value group. The prime ideals of the valuation ring $\mathcal{O}_v$ are linearly ordered by set-inclusion and have the following form:
\[
\mathfrak{m}^\Delta_v:=\{x\in K\mid vx>\delta,\text{ for all }\delta\in\Delta\},
\]
where $\Delta$ is a convex subgroup of $vK$ (see \cite[Lemma 2.3.1]{PE05}). The ideal $\mathfrak{m}_v:=\mathfrak{m}^{\{0\}}_v$, corresponding to the trivial convex subgroup $\{0\}$ of $vK$, is the unique maximal ideal of $\mathcal{O}_v$. The field $Kv$, defined as the quotient ring $\mathcal{O}_v/\mathfrak{m}_v$, is called the \emph{residue field} of $(K,v)$.

A \emph{homomorphism of valued fields} from $(K,v)$ to $(L,w)$ can be defined as a homomorphism of fields $\sigma:K\tto L$ such that $\sigma(\mathcal{O}_v)\subseteq\mathcal{O}_w$. The latter condition is sometimes phrased as \lq\lq$\sigma$ preserves the valuation\rq\rq. Since homomorphisms of valued fields are in particular homomorphisms of fields, they are automatically injective and will thus sometimes be called \emph{embeddings}. We say that $(L,w)$ is a \emph{valued field extension} of $(K,v)$ if $K\subseteq L$ and the inclusion map is an embedding of valued fields. In this way, valued fields and their homomorphisms form a category $\vFld$, which is a subcategory of $\Set$. By an \emph{isomorphism of valued fields} we mean an isomorphism in $\vFld$, namely, a bijective homomorphism of valued fields whose inverse (as a function) is an arrow in $\vFld$ too. This can be spelled out further as follows: a function $\sigma:K\tto L$ is an isomorphism of valued fields $(K,v)\simeq(L,w)$ if and only if it is an isomorphism of fields $K\simeq L$ such that $\sigma(\mathcal{O}_v)=\mathcal{O}_w$.  

Fix now a valued field $(K,v)$. There is a smallest ordinal\footnote{Such ordinal $\kappa$ is usually called the \emph{cofinality} of $vK$} $\kappa$ serving as the index set of a sequence $(\gamma_\nu)_{\nu<\kappa}$ that is cofinal in $vK$, i.e., such that for each $\delta\in vK$ there exists $\nu<\kappa$ such that $\delta<\gamma_\nu$. We say that a sequence $(x_\nu)_{\nu<\kappa}$ of elements of $K$ is a \emph{Cauchy sequence} if and only if 
\begin{quote}
\mbox{for every $\gamma\in vK$ there exists $\nu_0<\kappa$ such that if $\nu_0\leq\nu,\mu<\kappa$, then $v(x_\nu-x_\mu)>\gamma$.}
\end{quote}
A sequence $(x_\nu)_{\nu<\kappa}$ is, instead, said to be \emph{convergent} to an element $x$ belonging to some valued field extension $(L,w)$ of $(K,v)$ if and only if 
\begin{quote}
for every $\gamma\in wL$ there exists $\nu_0<\kappa$ such that if $\nu_0\leq\nu<\kappa$, then $w(x-x_\nu)>\gamma$. 
\end{quote}
If the latter property happens to hold, then we also say that the sequence $(x_\nu)_{\nu<\kappa}$ converges in $L$. If $(L,w)$ is a valued field extension of $(K,v)$, then we say that $K$ \emph{lies dense} in $L$ if every Cauchy sequence in $K$ converges in $L$, while $(K,v)$ is called \emph{complete} if and only if every Cauchy sequence in $K$ converges in $K$.

\begin{Fact}[Theorem 2.4.3 in \cite{PE05}]\label{completion}
Every valued field $(K,v)$ admits one and (up to isomorphism of valued fields) only one valued field extension $(K^c,v^c)$ -- called the \emph{completion} of $(K,v)$ -- which is complete and in which $K$ lies dense.
\end{Fact} 

An important consequence of the fact that $K$ lies dense in $K^c$ is that the value group $v^cK^c$ and the residue field $K^cv^c$ of $(K^c,v^c)$ are (canonically) isomorphic to $vK$ and $Kv$, respectively (cf.\ \cite[Proposition 2.4.4]{PE05}).

\subsection{Krasner hyperfields}

Hyperfields first appeared in \cite{Kra57}. In introducing them, Krasner was motivated by his interest for certain structures obtained from valued fields by means of the \lq\lq factor (or quotient) construction\rq\rq, which he himself described for the first time in the same article and later in \cite{Kra83}. For the algebraic definition of hyperfields we refer to \cite[Definition 2.7]{Lin23} and references therein. A more categorial treatment of these structures within the category of sets and (total) relations can be found in \cite{Lin23aims}.

The following definition of homomorphism for hyperfields has become standard in the literature.
\begin{Definition}
Let $(H,\boxplus,\cdot,0,1), (H',\boxplus',\cdot',0',1')$ be hyperfields. A function
\[
\sigma:H\tto H'
\]
is called a \emph{homomorphism of hyperfields} if $\sigma(0)=0'$, its restriction to the multiplicative groups is a homomorphism of groups and, in addition,
\begin{equation}\label{homohyp}
\sigma(x\boxplus y)\subseteq\sigma(x)\boxplus'\sigma(y)
\end{equation}
holds, for all $x,y\in H$.
\end{Definition}

Let us recall Krasner's factor construction which yields a projective system of hyperfields associated to any valued field (considered also in \cite{Lee20,LT22}). All hyperfields considered in this paper can be obtained via this construction (see \cite[Proposition 2.17]{Lin23}). Nevertheless, the statement \lq\lq all hyperfields are quotient\rq\rq\ is not valid in full generality \cite{Mas85}. Whether non-trivially valued hyperfields are automatically quotient or not is not known, as discussed in \cite{LT22}. 

For a valued field $(K,v)$ and an element $\gamma\in vK$ such that $\gamma\geq 0$, one considers the subgroup of the \emph{$1$-units of level $\gamma$} in $K^\times$:
\[
\mathcal{U}_v^\gamma:=\{u\in K\mid v(u-1)>\gamma\}.
\] 
It can be easily verified that $vu=0$ for all $u\in\mathcal{U}_v^\gamma$, so that the valuation map $v$ on $K$ factors through the quotient group $K^\times_\gamma:=K^\times/\mathcal{U}_v^\gamma$ and yields a map $v_\gamma:K_\gamma\tto vK\cup\{\infty\}$, where $K_\gamma:=K_\gamma^\times\cup\{[0]_\gamma\}$. In this paper, we follow the notation of \cite{Lee20} and we denote the multiplicative coset $x\mathcal{U}_v^\gamma$ of $x\in K$ in $K_\gamma$ as $[x]_\gamma$ (in particular, $[0]_\gamma=\{0\}$) and call the \emph{valued $\gamma$-hyperfield associated to} $(K,v)$ the hyperfield $(K_\gamma,\boxplus,\cdot,[0]_\gamma,[1]_\gamma)$, where 
\[
[x]_\gamma\boxplus [y]_\gamma:=\{[x+yu]_\gamma\mid u\in\mathcal{U}_v^\gamma \}\quad\text{ and }\quad [x]_\gamma\cdot [y]_\gamma:=[xy]_\gamma~.
\]
The use of the term \lq\lq valued\rq\rq\ is motivated once one observes that the map $v_\gamma$ satisfies (VAL1), (VAL2) and the following property analogous to (VAL3):
\begin{enumerate}
\item[(VAL3*)] $v_\gamma[z]_\gamma\geq\min\{v_\gamma[x]_\gamma,v[x]_\gamma\}$, for all $[x]_\gamma~, [y]_\gamma\in K_\gamma$ and all $[z]_\gamma\in [x]_\gamma\boxplus [y]_\gamma$
\end{enumerate} 
see also \cite{KLS22,PhD22,Lin23}. More generally, we shall call $(H,v)$ a \emph{valued hyperfield} whenever $v$ is a map from the hyperfield $H$ to an ordered abelian group $\Gamma$ (with the addition of $\infty$) satisfying (VAL1), (VAL2) and (VAL3*). These requirements are equivalent (cf.\ e.g.\ \cite[Lemma 3.4]{Lin23}) to $v$ being a homomorphism of hyperfields $H\tto\mathcal{T}(\Gamma)$, where $\mathcal{T}(\Gamma)$ denotes the \emph{generalised tropical hyperfield} associated to $\Gamma$:

\begin{Example}[Example 2.14 in \cite{Lin23}]\label{TGamma}
Let $\infty$ be a symbol such that $\gamma+\infty=\infty+\gamma=\infty>\gamma$ for all $\gamma\in\Gamma$. For $\gamma,\delta\in\Gamma\cup\{\infty\}$ we denote by $[\gamma,\delta]$ the closed interval containing all $\varepsilon\in\Gamma\cup\{\infty\}$ satisfying $\gamma\leq \varepsilon\leq \delta$. Then by setting $\gamma\boxplus \infty=\infty\boxplus \gamma=\{\gamma\}$ for all $\gamma\in\mathcal{T}(\Gamma)$ and:
\[
\gamma\boxplus \delta:=\begin{cases}\{\min\{\gamma,\delta\}\}&\text{if }\gamma\neq \delta,\\ [\gamma,\infty]&\text{if }\gamma=\delta.\end{cases}
\quad\quad(\gamma,\delta\in\mathcal{T}(\Gamma))
\]
It is not difficult to check that $(\mathcal{T}(\Gamma),\boxplus,+,\infty,0)$ is a hyperfield, called \emph{generalised tropical hyperfield associated to $\Gamma$}.

Note that, conversely, the order of $\Gamma$ can be recovered as follows:
\[
\gamma\leq\delta\iff \delta\in\gamma\boxplus\gamma.
\]
\end{Example}

In \cite[Section 3]{Lin23} the author shows that, as in the case of fields, the set $\mathcal{O}_v$ of the elements in a valued hyperfield $(H,v)$ with non-negative value under $v$ determines the valuation map (up to valuation-equivalence). As a consequence, homomorphisms of valued hyperfields are defined analogously to arrows in $\vFld$ and a category $\vHyp$ is thus obtained. A field $K$ with additive operation $+$ can be regarded as a hyperfield with additive operation $\boxplus$ defined as $x\boxplus y:=\{x+y\}$. Conversely, any hyperfield with a \emph{singlevalued} additive operation, i.e., such that $x\boxplus y$ is a singleton for all $x,y\in H$, can be regarded as a field. It is in this spirit that we view $\vFld$ as a subcategory of $\vHyp$. One can observe that the identity $1_{\mathcal{T}(\Gamma)}:\mathcal{T}(\Gamma)\tto \mathcal{T}(\Gamma)$ of the hyperfield $\mathcal{T}(\Gamma)$ in $\vHyp$ is a valuation on $\mathcal{T}(\Gamma)$. Thus, a valuation $v:H\tto \Gamma\cup\{\infty\}$ on a hyperfield $H$ is, equivalently, a $\vHyp$-arrow $(H,v)\tto(\mathcal{T}(\Gamma),1_{\mathcal{T}(\Gamma)})$.

Let us fix a valued field $(K,v)$. The following statement contains a number of fundamental properties of the valued $\gamma$-hyperfields associated to $(K,v)$, where $0\leq\gamma\in vK$. 

\begin{Lemma}[Lee Lemma]\label{Lee}
Take $x,y,x_0,\ldots,x_k\in K$ for some positive integer $k$ and let $\gamma\in vK$ be such that $\gamma\geq 0$. The following assertions hold:
\begin{enumerate}
\item[$(i)$] If $x\neq0$, then $[x]_\gamma=\{y\in K\mid v(x-y)> \gamma+vx\}$.
\item[$(ii)$] If $x$ and $y$ are not both $0$, then 
\[
\bigcup([x]_\gamma\boxplus[y]_\gamma)=\{z\in K\mid v(z-(x+y))>\gamma+\min\{vx,vy\}\}.
\]
\item[$(iii)$] If $x$ and $y$ are not both $0$, then 
\[
0\in \bigcup([x]_\gamma\boxplus[y]_\gamma)\iff\bigcup([x]_\gamma\boxplus [y]_\gamma)=\bigl\{z\in K\mid vz>\gamma+\min\{vx,vy\}\bigr\}.
\]
\item[$(iv)$] $[x_0+\ldots+x_k]_\gamma\in [x_0]_\gamma\boxplus\ldots\boxplus[x_k]_\gamma$ .
\item[$(v)$] If $x_0,\ldots,x_k\in\mathcal{O}_v$ are not all $0$, then
\[
[y]_\gamma\in [x_0]_\gamma\boxplus\ldots\boxplus[x_k]_\gamma\quad\implies\quad v(x_0+\ldots+x_k-y)>\gamma.
\]
\end{enumerate}
\end{Lemma}

\begin{proof}
See \cite[Lemma 3.1]{Lee20} or \cite[Lemma 3.3]{PhD22}.
\end{proof}

From the above fundamental properties we now wish to isolate the following important consequences.

\begin{Proposition}\label{single}
Take $[x]_\gamma~,[y]_\gamma\in K_\gamma$, where $( K_\gamma,v_\gamma)$ denotes the valued $\gamma$-hyperfield  of $(K,v)$ for some $\gamma\in vK$ such that $\gamma\geq 0$. Then all elements of $[x]_\gamma\boxplus[y]_\gamma$ have the same value under $v_\gamma$, unless $[0]_\gamma\in[x]_\gamma\boxplus[y]_\gamma$.
\end{Proposition}

\begin{proof}
See \cite[\S\ 3]{Kra57} or \cite[Proposition 3.19]{PhD22}.
\end{proof}

The above proposition permits to induce from $v_\gamma$ an ultrametric distance on $K_\gamma$ which we denote by $d_\gamma$ (see e.g.\ \cite[Definition 4.1]{Lin23}).

\begin{Proposition}\label{balls}
Let $( K_\gamma,v_\gamma)$ be the valued $\gamma$-hyperfield of a valued field $(K,v)$, where $0\leq \gamma\in vK$. If $x$ and $y$ are not both $0$, then $[x]_\gamma\boxplus[y]_\gamma$ is the open ultrametric ball of radius $\delta:=\gamma+\min\{vx,vy\}$ around $[x+y]_\gamma\in K_\gamma$ with respect to $d_\gamma$.
\end{Proposition}

\begin{proof}
See \cite[\S\ 3]{Kra57} or \cite[Proposition 3.26]{PhD22}.
\end{proof}

\section{Main results}

For the rest of the paper:
\begin{itemize}
\item $(K,v)$ denotes a valued field.
\item $(\Gamma,\leq,+,0)$ denotes an ordered abelian group containing $(vK,\leq,+,0)$ as a substructure in the language $\{\leq,+,0\}$ of ordered groups.
\end{itemize}

In the first result of this final section, we highlight another consequence of the fact that a valued field lies dense in its completion. As usual, if a homomorphism of hyperfields $\sigma$ is bijective and its inverse is a homomorphism of hyperfields as well, then $\sigma$ is called an \emph{isomorphism of hyperfields}.

\begin{Lemma}\label{complet}
Let $(K^c,v^c)$ be the completion of $(K,v)$ and identify $K$ and $vK$ with the subsets of $K^c$ and $v^cK^c$ to which they are canonically isomorphic. Then for all $\gamma\in vK$ such that $\gamma\geq 0$, there is an isomorphism of hyperfields $\sigma_\gamma:K_\gamma\tto K^c_\gamma$ such that $v_\gamma(\sigma_\gamma(a))=v^c_\gamma(a)$ holds, for all $a\in K^c_\gamma$.
\end{Lemma}
\begin{proof}
Fix $\gamma\in vK$ such that $\gamma\geq0$. Just for this proof, we will denote by $[x]^c$ the class of $x\in K^c$ in $K^c_\gamma$ and by $[y]$ the class of $y\in K$ in $K_\gamma$. It follows from Lemma \ref{Lee} $(i)$ that, for all nonzero $x\in K^c$, we have that $[x]^c$, as a subset of $K^c$, is an open ultrametric ball (with respect to the ultrametric induced on $K^c$ by $v^c$). Since $K$ lies dense in $K^c$, there is $y\in K$ such that $y\in[x]^c$. On the other hand, $[x]^c$ is an equivalence class in $K^c$ with respect to an equivalence relation whose restriction to $K$ has $[y]$ among its equivalence classes. Consequently, $[y]^c= [x]^c$ as subsets of $K^c$ and if $y'\in K$ satisfies $y'\in[x]^c$ as well, then $[y]^c= [x]^c=[y']^c$ and thus $[y]=[y']$ must hold. Since all $x\in K$ belong to the class $[x]^c$ in $K^c_\gamma$, this proves that the assignment $[x]\mapsto [x]^c$ defines a bijective function $\sigma_\gamma:K_\gamma\tto K^c_\gamma$. From the definitions and the inclusion $\mathcal{U}_v^\gamma\subseteq\mathcal{U}_{v^c}^\gamma$ it easily follows at this point that $\sigma_\gamma$ is an isomorphism of hyperfields satisfying the assertion of the lemma.
\end{proof}

The next result is that the valued $\gamma$-hyperfields of $(K,v)$, form a diagram:
\begin{equation}\label{diagramvK}
\left(\begin{tikzcd}[column sep=huge, row sep=small]
K_\delta\arrow[dd,"\rho_{\delta,\gamma}"']\arrow[dr,"v_\delta"]&\\
&\mathcal{T}(\Gamma)\\
K_\gamma\arrow[ur,"v_\gamma"']&
\end{tikzcd}\right)_{\delta\geq\gamma\geq0}
\end{equation}
(of shape the poset-category associated to $vK_{\geq 0}$) in the slice category $\vHyp/\mathcal{T}(\Gamma)$.

\begin{Lemma}
If $\gamma,\delta\in vK\subseteq\Gamma$ satisfy $0\leq\gamma\leq\delta$, then the functions
\[
\begin{tikzcd}[column sep=huge,row sep=tiny]
\rho_{\delta,\gamma}:K_\delta\arrow[r]&K_\gamma\\
{[}x{]}_\delta\arrow[r,mapsto]&{[}x{]}_\gamma
\end{tikzcd}
\]
are arrows in the slice category $\vHyp/\mathcal{T}(\Gamma)$. Furthermore, if $\gamma\leq\delta\leq\varepsilon$ are non-negative elements of $vK$, then
\[
\rho_{\varepsilon,\gamma}=\rho_{\delta,\gamma}\circ\rho_{\varepsilon,\delta}~.
\]
\end{Lemma}
\begin{proof}
First we show that $\rho_{\delta,\gamma}$ is well-defined for all $\gamma,\delta$ as in the statement. To this end, assume that $[x]_\delta=[y]_\delta$~. Then there exist $t\in\mathcal{U}_v^\delta$ such that $x=yt$. Since $\gamma\leq\delta$ we have that $\mathcal{U}_v^\delta\subseteq\mathcal{U}_v^\gamma$, so we obtain that $x=yt$ for some $t\in\mathcal{U}_v^\gamma$ and thus
\[
\rho_{\delta,\gamma}([x]_\delta)=[x]_\gamma=[y]_\gamma=\rho_{\delta,\gamma}([y]_\delta).
\]
It is clear that $\rho_{\delta,\gamma}[0]_\delta=[0]_\gamma$ holds. Furthermore, the following computation:
\[
\rho_{\delta,\gamma}([x]_\delta[y]^{-1}_\delta)=\rho_{\delta,\gamma}([xy^{-1}]_\delta)=[xy^{-1}]_\gamma=[x]_\gamma[y]^{-1}_\gamma=\rho_{\delta,\gamma}([x]_\delta)\rho_{\delta,\gamma}([y]_\delta)^{-1},
\]
for all $x,y\in K$ with $y\neq 0$, shows that the restriction of $\rho_{\delta,\gamma}$ to $K_\delta^\times$ is a homomorphism of groups (with codomain $K_\gamma^\times$). In addition, we have that
\begin{align*}
\rho_{\delta,\gamma}([x]_\delta\boxplus_\delta [y]_\delta)&=\{\rho_{\delta,\gamma}([x+yt]_\delta)\mid t\in\mathcal{U}_v^\delta\}\\
&=\{[x+yt]_\gamma\mid t\in\mathcal{U}_v^\delta\}\\
&\subseteq\{[x+yt]_\gamma\mid t\in\mathcal{U}_v^\gamma\}\\
&=[x]_\gamma\boxplus_\gamma [y]_\gamma\\
&=\rho_{\delta,\gamma}([x]_\delta)\boxplus_\gamma\rho_{\delta,\gamma}([y]_\delta).
\end{align*}
where $\boxplus_\delta$ and $\boxplus_\gamma$ denote the additive operation of $K_\delta$ and $K_\gamma$, respectively and we used again the fact that $\mathcal{U}_v^\delta\subseteq\mathcal{U}_v^\gamma$. We have proved that $\rho_{\delta,\gamma}$ is a homomorphism of hyperfields.

We deduce that $\rho_{\delta,\gamma}$ is an arrow in $\vHyp/\mathcal{T}(\Gamma)$ by noticing that the following chain of equalities: 
\[
v_\delta[x]_\delta=vx=v_\gamma[x]_\gamma=v_\gamma(\rho_{\delta,\gamma}[x]_\delta)
\]
holds, for all $[x]_\delta\in K_\delta$, by the definition of the valuations $v_\gamma$ and $v_\delta$. The last assertion of the lemma follows immediately from the definition of the functions $\rho_{\delta,\gamma}$ ($\delta\geq\gamma\geq0$).
\end{proof}

The assignment $x\mapsto[x]_\gamma$ defines a function $\rho_\gamma: K\tto K_\gamma$, for all non-negative $\gamma\in vK$. It follows from the definitions that these functions are homomorphisms of valued hyperfields such that $vx=v_\gamma[x]_\gamma$ for all $x\in K$, i.e., the following triangular diagrams:
\[
\left(\begin{tikzcd}[column sep=huge, row sep=small]
K\arrow[dd,"\rho_\gamma"']\arrow[dr,"v"]&\\
&\mathcal{T}(\Gamma)\\
K_\gamma\arrow[ur,"v_\gamma"']&
\end{tikzcd}\right)_{\gamma\geq0}
\]
commute in $\vHyp$. Therefore, the functions $\rho_\gamma$ are arrows in the slice category $\vHyp/\mathcal{T}(\Gamma)$. Moreover, they respect the functions $\rho_{\delta,\gamma}$ in the sense that, for all non-negative $\gamma,\delta\in vK$ we have that the following diagram:
\[
\begin{tikzcd}[column sep=huge, row sep=small]
K_\delta\arrow[dd,"\rho_{\delta,\gamma}"]&\\
&K\arrow[ul,"\rho_\gamma"']\arrow[dl,"\rho_\delta"]\\
K_\gamma&
\end{tikzcd}
\]
commutes in $\vHyp$. The above discussion shows that $(K,v)$ is the vertex of a cone over the diagram \eqref{diagramvK} in $\vHyp/\mathcal{T}(\Gamma)$, i.e., the following diagram:
\[
\begin{tikzcd}[column sep=huge, row sep=small]
K_\delta\arrow[dd,"\rho_{\delta,\gamma}"]\arrow[rrd,bend left,"v_\delta"]&&\\
&K\arrow[ul,"\rho_\gamma"']\arrow[dl,"\rho_\delta"]\arrow[r,"v"]&\mathcal{T}(\Gamma)\\
K_\gamma\arrow[rru,bend right,"v_\gamma"]&&
\end{tikzcd}
\]
is commutative in $\vHyp$. Now consider the completion $(K^c,v^c)$ of $(K,v)$. If, as before, we identify $vK$ with the subset of $v^cK^c$ to which it is canonically isomorphic, then from Lemma \ref{complet} we deduce that $K^c$ too is the vertex of a cone over the same diagram \eqref{diagramvK}. In addition, $K$ embeds as a valued field in $K^c$ by Fact \ref{completion} and such an embedding can be seen to be an arrow in $\vHyp/\mathcal{T}(\Gamma)$. Before moving forward let us prove the following useful lemma, which states that (with the right choice of $\Gamma$) all cones in $\vHyp/\mathcal{T}(\Gamma)$ over the diagram \eqref{diagramvK} are (valued) fields.

\begin{Lemma}\label{FVhyp}
Let $(H,w)$ be a valued hyperfield such that there are order-preserving group-embeddings $vK\hookrightarrow wH\hookrightarrow \Gamma$ and assume that $(H,w)$ is the vertex of a cone in $\vHyp/\mathcal{T}(\Gamma)$ over the diagram \eqref{diagramvK}. Then $H$ is a field, i.e., for all $x,y\in H$ we have that $x\boxplus y$ is a singleton, where $\boxplus$ denotes the additive operation of the hyperfield $H$.
\end{Lemma}

\begin{proof}
We work up to the given embeddings of ordered abelian groups. Fix $\gamma\in vK$ such that $\gamma\geq 0$. We denote by $f_\gamma:H\tto  K_\gamma$ the sides of the given cone in $\vHyp/\mathcal{T}(\Gamma)$. Pick $x,y\in H^\times$ and $z,z'\in x\boxplus y$. We claim that $0\in z'\boxplus z^-$, where $z^-$. Since $f_\gamma$ is an an arrow in the slice category $\vHyp/\mathcal{T}(\Gamma)$, we obtain that $f_\gamma(z),f_\gamma(z')\in f_\gamma(x)\boxplus_\gamma f_\gamma(y)$ holds in $ K_\gamma$, where $\boxplus_\gamma$ denotes the additive operation of the hyperfield $K_\gamma$. In addition, we obtain that the equalities $wx=v_\gamma f_\gamma(x)$ and $wy=v_\gamma f_\gamma(y)$ hold in $\Gamma$. Thus, by Proposition \ref{balls}, we have that $f_\gamma(x)\boxplus_\gamma f_\gamma(y)$ is an open ultrametric ball in $K_\gamma$ of radius $\gamma+\min\{wx,wy\}$. Let now $(\gamma_\nu)_{\nu<\kappa}$ be an increasing and cofinal sequence of non-negative elements of $vK$ and take an arbitrary $\delta\in vK$. We consider some $\nu<\kappa$ which is large enough so that $\gamma_\nu+\min\{wx,wy\}>\delta$ holds in $\Gamma$. If we suppose that $[0]_{\gamma_\nu}\notin f_{\gamma_\nu}(z')\boxplus_{\gamma_\nu} f_{\gamma_\nu}(z)^-$, then by Proposition \ref{single} and since $f_{\gamma_\nu}$ is a homomorphism of hyperfields, we obtain that for any $a\in z'-z^-$, the value $wa=v_{\gamma_\nu} f_{\gamma_\nu}(a)\in vK\cup\{\infty\}$ is larger than $\delta$. Since $\delta$ is arbitrary in $vK$, this implies that $wa=\infty$ and so $a=0$. We deduce that
\[
[0]_{\gamma_\nu}=f_{\gamma_\nu}(a)\in f_{\gamma_\nu}(z')\boxplus_{\gamma_\nu}f_{\gamma_\nu}(z)^-.
\]
This contradiction shows that $[0]_{\gamma_\nu}\in f_{\gamma_\nu}(z')\boxplus_{\gamma_\nu}f_{\gamma_\nu}(z)^-$ must hold in $K_{\gamma_\nu}$ and, as a consequence, we obtain that $f_{\gamma_\nu}(z')=f_{\gamma_\nu}(z)$. Furthermore, since $f_{\gamma_\nu}$ is an arrow in $\vHyp/\mathcal{T}(\Gamma)$, we have that $wz'=wz$ and by enlarging $\nu$ (if necessary) we can ensure that $\gamma_{\nu}+wz>\delta$ as well. Now, for all $a\in z'\boxplus z^-$ we obtain that $f_{\gamma_\nu}(a)\in f_{\gamma_\nu}(z')\boxplus_{\gamma_\nu} f_{\gamma_\nu}(z)^-$ and, again by Proposition \ref{balls}, $f_{\gamma_\nu}(z')\boxplus_{\gamma_\nu} f_{\gamma_\nu}(z)^-$ is an open ultrametric ball of radius $\gamma_\nu+wz$ and center $[0]_{\gamma_\nu}$. Therefore, $wa=v_{\gamma_\nu}f_{\gamma_\nu}(a)\in vK\cup\{\infty\}$ will be larger than $\delta$ and since $\delta$ is arbitrary in $vK$, it follows that $a=0$ anyway. At this point our claim is proved and $z'=z$ follows. The proof is complete.
\end{proof}

By the following theorem the valued field extensions of a valued field which embed in its completion are characterised in terms of the diagram \eqref{diagramvK}.

\begin{Theorem}\label{main}
Let $(L,w)$ be a valued field extension of $(K,v)$ such that there is an order-preserving group-embedding $wL\hookrightarrow\Gamma$. Then the following statements are equivalent: 
\begin{enumerate}
\item[$(i)$] $(L,w)$ embeds as a valued field into $(K^c,v^c)$.
\item[$(ii)$] For all $\gamma\in vK$ such that $\gamma\geq 0$, there is an isomorphism in $\vHyp/\mathcal{T}(\Gamma)$:
\[
\begin{tikzcd}
\sigma_\gamma:(L_\gamma,w_\gamma)\arrow[r]&(K_\gamma,v_\gamma).
\end{tikzcd}
\]
\end{enumerate}
\end{Theorem}

\begin{proof}
We begin by proving that $(i)$ implies $(ii)$. Since $L$ contains $K$ which lies dense in $K^c$, it follows that $L$ lies dense in $K^c$ too. Hence $(ii)$ follows as in the proof of Lemma \ref{complet}.

For the implication from $(ii)$ to $(i)$, a little more effort is needed. First, we need to fix an increasing and cofinal sequence of non-negative elements $(\gamma_\nu)_{\nu<\kappa}$ in the value group $vK$, as we have done in the proof of Lemma \ref{FVhyp}. Then, for any $x\in L^\times$ and all $\nu<\kappa$, we set $y_\nu\in K^\times$ to be a representative for the class $\sigma_{\gamma_\nu}([x]_{\gamma_{\nu}})\in K_{\gamma_\nu}^\times$. By the assumption $(ii)$ and the definition of the hyperfield valuations $v_{\gamma_\nu}$ and $w_{\gamma_\nu}$, we deduce that $vy_\nu=wx$ holds in $\Gamma$ for all $\nu<\kappa$. In addition, since $(\gamma_\nu)_{\nu<\kappa}$ is increasing, Lemma \ref{Lee} $(i)$ yields that
\[
v(y_\nu-y_\mu)>\gamma_\nu+wx
\]
holds in $\Gamma$, for all $\nu<\mu<\kappa$. Now, by the cofinality of $(\gamma_\nu)_{\nu<\kappa}$ in $vK$, the latter inequality implies that $(y_\nu)_{\nu<\kappa}$ is a Cauchy sequence in $K$ which then converges to some $y\in K^c$. 

We claim that $y$ does not depend on the choice of the representatives $y_\nu\in K$, but only on the class $\sigma([x]_{\gamma_{\nu}})\in K_\gamma^\times$. For let $y'_\nu\in K$ be such that $[y'_\nu]_{\gamma_\nu}=\sigma_{\gamma_\nu}([x]_{\gamma_{\nu}})$ be another choice. As above, $(y'_\nu)_{\nu<\kappa}$ is a Cauchy sequence in $K$ and we denote by $y'$ its limit in $K^c$. If $\delta\in vK$ and $\nu<\kappa$ are such that $\gamma_{\nu}+wx>\delta$ and $v^c(y-y_\nu),v^c(y'_\nu-y')>\delta$ hold in $\Gamma$, then, by Lemma \ref{Lee} $(i)$ and since $vy'_\nu=wx=vy_\nu$, we obtain that
\[
v(y_\nu-y'_\nu)>\gamma_\nu+wx>\delta
\]
and then
\begin{align*}
v^c(y-y')&=v^c(y-y_\nu+y_\nu-y'_\nu+y'_\nu-y')\\
&\geq\min\{v^c(y-y_\nu),v^c(y_\nu-y'_\nu),v^c(y'_\nu-y')\}>\delta
\end{align*}
hold in $\Gamma$. Since $\delta$ is arbitrary in $vK=vK^c$, we may conclude that $v^c(y-y')=\infty$, i.e., $y'=y$ holds in $K^c$.

Our next claim is that the assignment $x\mapsto y$ defines an embedding of valued fields 
\[
\begin{tikzcd}[sep=huge]
\sigma:(L,w)\arrow[r]& (K^c,v^c).
\end{tikzcd}
\]
To see why this holds, most of the efforts are devoted to the verification that $\sigma$ preserves the additive operations. For take $x,y\in L$ and assume without loss of generality that $wx\leq wy$. If $z_\nu,a_\nu,b_\nu\in K$ are such that $[z_\nu]_{\gamma_\nu}=[x+y]_{\gamma_\nu}$, $[a_\nu]_{\gamma_\nu}=[x]_{\gamma_\nu}$ and $[b_\nu]_{\gamma_\nu}=[y]_{\gamma_\nu}$ for all $\nu<\kappa$, then, as before, these elements form Cauchy sequences in $K$. Let us then denote by $z$, $a$ and $b$ their limits in $K^c$, respectively. By definition of $\sigma$, we have that $\sigma(x+y)=z$, $\sigma(x)=a$ and $\sigma(y)=b$. Our aim now will be to prove that $z=a+b$ holds in $K^c$. We first obtain from Lemma \ref{Lee} $(iv)$ that
\[
z_\nu\in\bigcup([a_\nu]_{\gamma_\nu}+[b_\nu]_{\gamma_\nu})
\]
holds, for all $\nu<\kappa$. Therefore, if we fix any $\delta\in vK$ and let $\nu<\kappa$ be large enough so that $\gamma_{\nu}+wx>\delta$ and 
\[
v^c(a-a_\nu),~v^c(b-b_\nu),~v^c(z-z_\nu)>\delta
\] 
hold in $\Gamma$, then an application of Lemma \ref{Lee} $(ii)$ yields that
\[
v(a_\nu+y_\nu-z_\nu)>\gamma_\nu+wx>\delta
\]
holds in $\Gamma$, where we used the fact that $va_\nu=wx$ for all $\nu<\kappa$. Thus, we obtain that
\begin{align*}
v^c(a_\nu+b_\nu-z)&=v^c(a_\nu+b_\nu-z_\nu+z_\nu-z)\\
&\geq\min\{v^c(a_\nu+b_\nu-z_\nu),v^c(z_\nu-z)\}>\delta
\end{align*}
and, consequently,
\begin{align*}
v^c(a+b-z)&=v^c(a-a_\nu+b-b_\nu+a_\nu+b_\nu-z)\\
&\geq\min\{v^c(a-a_\nu),v^c(b-b_\nu),v^c(a_\nu+b_\nu-z)\}>\delta
\end{align*}
hold in $\Gamma$. Since $\delta$ is arbitrary in $vK$ and $v^c(a+b-z)\in vK\cup\{\infty\}$, we deduce that $v^c(a+b-z)=\infty$ and, consequently, $z=a+b$, as contended.

Now, it suffices to recall that $[-x]_\gamma=-[x]_\gamma$ and that if $y\neq0$, then $[xy^{-1}]_\gamma=[x]_\gamma[y]^{-1}_\gamma$ hold in $L_\gamma$ to immediately deduce that $\sigma(-x)=\sigma(x)^-$ and that $\sigma(ab^{-1})=\sigma(a)\sigma(b)^{-1}$ must hold in $K^c_\gamma$. We have proved that $\sigma$ is a homomorphism of fields and, as such, an embedding.

Finally, since $vy_\nu=wx$ holds in $\Gamma$, for all $x\in L^\times$ and all $\nu<\kappa$ (as we have already shown above) it can be easily verified from the definition of convergent sequences, that the element $\sigma(x)\in K^c$, to which the Cauchy sequence $(y_\nu)_{\nu<\kappa}$ in $K$ converges, satisfies $v^c(\sigma(x))=vy_\nu$, for all $\nu<\kappa$. We conclude that $v(\sigma(x))=wx$ holds, for all $x\in L$. In particular, we have that $\sigma$ is an embedding of valued fields.
\end{proof}

In the proof of the implication from $(ii)$ to $(i)$ in the above theorem, we have used the assumption that $(L,w)$ is an extension of $(K,v)$ only for identifying $vK$ and $v^cK^c$ with a canonical subset of $wL$. However, in the final analysis, this identification is not necessary and is performed only for a smoother exposition of the reasonings in the proof.

\begin{Scholium}[\footnote{The term \lq scholium\rq\ (literally, a marginal note) is used e.g.\ in \cite{Joh02} to denote something which follows directly from the \emph{proof} of a preceding result, as opposed to a corollary which follows directly from the \emph{statement} of the preceding result.}]\label{sch}
If $(L,w)$ is any valued field such that there is an order-preserving group-embedding $wL\hookrightarrow \Gamma$ and, up to this embedding, condition $(ii)$ of Theorem \ref{main} holds, then there is a $\vHyp/\mathcal{T}(\Gamma)$-arrow $(L,w)\tto(K^c,v^c)$.
\end{Scholium}

Now, under the assumptions of the above result, for a valued hyperfield $(L,w)$, we deduce, first, that $L$ is a field by Lemma \ref{FVhyp}. Then, we denote by $f_\gamma:L\to L_\gamma\simeq K_\gamma$ and by $\tilde{\rho}_\gamma:K^c\to K^c_\gamma\simeq K_\gamma$ the projections onto the valued $\gamma$-hyperfields associated to $(K,v)$ of $(L,w)$ and $(K^c,v^c)$, respectively. It is so straightforward to verify that the embedding $\sigma$ that we have constructed in the proof of Theorem \ref{main} is unique with the property that $f_\gamma=\tilde{\rho}_\gamma\circ\sigma$ for all $\gamma\in vK$ such that $\gamma\geq 0$. Indeed, this conclusion follows from the fact that for any $x\in L^\times$, the classes $([x]_\gamma)_{\gamma\in vK_{\geq0}}$ form a chain of ultrametric balls in $L$ of increasing radii and, moreover, the set of this radii is cofinal in $vK$. We have now fully proved Krasner's result \cite[\S\ 4]{Kra57} in a purely categorial language:

\begin{Theorem}
The completion $(K^c,v^c)$ of $(K,v)$ is ($\vHyp$-isomorphic to) the limit of the diagram \eqref{diagramvK} in $\vHyp/\mathcal{T}(\Gamma)$, for any ordered abelian group extension $\Gamma$ of $vK$.
\end{Theorem}

We note moreover, that in the above setting the assumption $\Gamma=vK=wL$ is \textit{a posteriori} not restrictive.

\begin{Corollary}
    $(K,v)$ is complete if and only if it is the source of the limit of the diagram \eqref{diagramvK}, with $\Gamma=vK$, in $\vHyp/\mathcal{T}(\Gamma)$.
\end{Corollary}

\section{Conclusion and future work}

It is possible (cf.\ \cite[Proposition 1.17]{LT22}) to associate a diagram of the form \eqref{diagramvK} as above also in case the additive operation of $(K,v)$ is not assumed to be singlevalued, i.e., starting with any valued hyperfield $(K,v)$. As a corollary to our approach, we deduce that if among the cones over the diagram \eqref{diagramvK} there is one whose vertex is a valued field, then the limit exists in $\vHyp$ and is up to isomorphism the valuation-theoretic completion of the latter valued field. In particular, this happens under weaker assumptions on $(K,v)$ than that $K$ is a field, e.g., by \cite[Proposition 1.27]{LT22}, it suffices that $v$ induces an ultrametric on $K$ in the way originally required by Krasner. The problem whether the just mentioned assumption on $(K,v)$ is also necessary or can be weakened further is outside the scope of this work and left open for future investigations. 

Finally, a characterisation of generalised tropical hyperfields in purely category-theoretic terms becomes of interest. A characterisation is given in \cite[Theorem 5.2]{Lin23} which relates the solution to the latter problem to a categorial description of the broader class of  \emph{stringent} hyperfields (in the sense of \cite{BS21}).

\bibliography{submittedbiblio}
\bibliographystyle{acm}

\end{document}